\documentclass[12pt,reqno]{amsart}

\usepackage{amsfonts,amsmath,amsthm}
\usepackage{amssymb,epsfig}
\usepackage{ascmac}  
\usepackage{cases, verbatim}
\usepackage[usenames]{color}
\usepackage{esint}
\usepackage{setspace}

\usepackage{enumerate}

\usepackage{color} 
\definecolor{vert}{rgb}{0,0.6,0}

\usepackage[sort, numbers]{natbib}
\theoremstyle{plain}
\newtheorem{thm}{Theorem}

\newtheorem{defn}{Definition}

\newtheorem{lem}[thm]{Lemma}

\theoremstyle{remark}
\newtheorem{rem}{\bf{Remark}}

\newcommand{\Ha}{\mathcal{H}^{s}}

\DeclareMathOperator*{\osc}{osc}

\begin{document}
\title[Uniformly continuous Sobolev functions in the critical case]
{Sobolev functions in the critical case are uniformly continuous in $s$-Ahlfors regular metric spaces when $s\le 1$}

\author[X. Zhou]{Xiaodan Zhou}
\address{Xiaodan Zhou, Department of Mathematics, University of Pittsburgh, Pittsburgh, PA 15260, USA, {\tt xiz78@pitt.edu}}

\thanks{The author was partially supported by NSF grant DMS-1500647 of
Piotr Haj\l{}asz.}

\date{}

\maketitle
\begin{abstract}
We prove that functions in the Haj\l{}asz-Sobolev space $M^{1,s}$ on an $s$-Ahlfors regular metric space are uniformly continuous when $s\le 1$. 
\end{abstract}
\section{Introduction}

Sobolev functions $u\in W^{1,1}([a,b])$ are absolutely continuous. When $n\ge 2$, $u(x)=\log|\log|x||$ provides an example of a discontinuous Sobolev function in $ W^{1,n}(B^n(0, e^{-1}))$. 
The aim of this paper is to investigate continuity of the Sobolev functions in the critical case, i.e. $p=n$ on more general metric spaces including fractals.

\begin{defn}
Let $(X, d, \mu)$ be a metric space equipped with a Borel measure $\mu$ and let $s>0$. We say that a metric measure space $(X,d,\mu)$ is $s$-Ahlfors regular if there is a constant $C_A\ge 1$, such that
\begin{equation*}\label{ahlfors}
C^{-1}_Ar^s\le\mu(B(x,r))\le C_Ar^s, \quad{\rm for}\quad x\in X\quad{\rm and}\quad 0<r<{\rm diam}(X).
\end{equation*}

\end{defn}
An important class of Ahlfors regular spaces is provided by fractals. Hutchinson \cite{Hu}, proved that fractals generated by iterated function systems satisfying the open set condition are Ahlfors regular.

 If $X$ is $s$-Ahlfors regular with respect to $\mu$, then $\mu$ is comparable to the Hausdorff measure $\mathcal{H}^s$. So $X$ is $s$-Ahlfors regular with respect to $\mathcal{H}^s$ \cite[Exercise 8.11.]{Hei1}.

Analysis on fractals has developed a great deal since eighties \cite{Ba, DS, JW, Kig2, Str} and the theory of Sobolev spaces on metric spaces play an important role in this development.
There are various extensions of the theory of Sobolev spaces to metric-measure spaces \cite{AGS, AT, Ch1, Ha1,Ha2, HK1, Hei1, HKST, Na1}. While the approach based on upper gradients (Cheeger and Newtonian Sobolev spaces) requires the space to be highly connected, it does not apply to fractals with limited connectivity properties. On the other hand, the theory of Haj\l{}asz-Sobolev spaces is rich without any assumption on connectivity of the space, and Haj\l{}asz-Sobolev spaces on fractals have been investigated in \cite{Hu1, Hu2, KPAS1, Ris, Ya}.

\begin{defn}
Let $(X, d, \mu)$ be a  metric space equipped with a Borel measure $\mu$. 
For $0<p<\infty$ we define the Haj\l{}asz-Sobolev space $M^{1,p}(X,d, \mu)$ to be the set of all functions $u\in L^p(X)$ for which there exists nonnegative Borel functions $g\in L^p(X)$ such that
\begin{equation*}
|u(x)-u(y)|\le d(x,y)\big(g(x)+g(y)\big)\ \ \ \mu-a.e.
\end{equation*}
Denote by $D(u)$ the class of all nonnegative Borel functions $g$ that satisfy the above inequality. Thus $u\in M^{1,p}(X,d, \mu)$ if and only if $u\in L^p(X)$ and $D(u)\cap L^p\neq \emptyset$. The space $M^{1,p}(X,d, \mu)$ is linear and we equip it with
\[
||u||_{M^{1,p}}=||u||_{L^p}+\inf_{g\in D(u)}||g||_{L^p}.
\]
\end{defn}
\begin{rem}
If $p\ge 1$, $||\cdot||_{M^{1,p}}$ is a norm and $M^{1,p}$ is a Banach space \cite{Ha1}. For applications of Haj\l{}asz-Sobolev spaces with $0<p<1$, see for example \cite{Ha2, JY, KS, KYZ1, KYZ2, Zh}. 
\end{rem}

The following result is the main theorem of this paper.
\begin{thm}
Let $(X, d, \mathcal{H}^s)$ be an $s$-Ahlfors regular metric space and $0<s\le 1$. If $u\in M^{1,s}(X, d, \mathcal{H}^{s})$, then $u$ is uniformly continuous. Moreover, there exists a constant $C>0$, such that for any ball $B\subset X$,
\begin{equation}\label{0}
\osc_{B}|u|=\sup_{x,y\in B}|u(x)-u(y)|\le C\Big(\int_{2B}g^s d\Ha\Big)^{\frac{1}{s}}\tag{$\ast$},
\end{equation}

where $g\in D(u)\cap L^s(X)$.
\end{thm}
Here and in what follows by $2B$ we denote a ball concentric with $B$ and with twice the radius.
\begin{rem}
For example, the ternary Cantor set $C$ equipped with the Euclidean distance is $s$-Ahlfors regular with $s=\frac{\log 2}{\log 3}$, \cite[Example 2.7]{Fa1}.  
\end{rem}

\begin{rem}
The Sobolev functions in $W^{1,1}([a,b])$ are absolutely continuous. More generally, 
if the space $X$ is compact, connected and $\mathcal{H}^1(X)<\infty$, the functions in the Newtonian Sobolev spaces $N^{1,1}(X ,d ,\mu)$ are also absolutely continuous \cite{XZ1}.

\end{rem}

\begin{rem}
The restriction that $s\le1$ is necessary in our statement, since we apply the reverse Minkowski inequality in the proof. We conjecture that there alway exists discontinuous Sobolev functions in $M^{1,s}(X)$ when $X$ is an $s$-Ahlfors regular space with $s>1$.
\end{rem}
\section{Proof of the main theorem}
Let $C$ denote a general constant whose value can change even in the same string of estimates. 

\begin{lem}\label{BS}

Let $X$ be an $s$-Ahlfors regular space. Then there is a constant $0<C_0<1$ such that for any ball $B(x,r)$ with $0<r<{\rm diam}(X)$,
there is a sequence of balls $\{B_i\}_{i=1}^{\infty}=\{B(x_i, C_0^{i}r)\}_{i=1}^{\infty}\subset B(x, r)$ satisfying
\begin{itemize}
\item[(1)] $B_i\cap B_j=\emptyset$, if $i\neq j$,\\
\item[(2)] 
$B(x_i, C_0^{i}r)\subset B(x, C_0^{i-1}r)\setminus B(x, C_0^{i}r)$.\\
\end{itemize}
\end{lem}
\begin{proof}
Since $X$ is $s$-Ahlfors regular, there is a constant $C_A\ge 1$, such that for any ball $B(x, r)\subset X$ with $0<r<{\rm diam}(X)$, we have
\begin{equation*}\label{ahlfors}
C^{-1}_Ar^s\le\Ha(B(x,r))\le C_Ar^s.
\end{equation*}
We claim that for any $0<c<C_A^{-2/s}$, $$B(x, r)\setminus B(x, cr)\neq \emptyset.$$ 
For if not, $B(x, r)\setminus B(x, cr)=\emptyset$ would imply
\[
C_A^{-1}r^s\le\Ha(B(x,r))=\Ha(B(x, cr))\le C_A(cr)^s
\]
so $C_A^{-2/s}\le c$, which is a contradiction. Fix $0<c<C_A^{-2/s}$ and let $$C_1=\min\{\frac{c}{2}, \frac{\sqrt{5}-1}{2}\}.$$ Clearly, $0<C_1<1$. 

Fix a ball $B(x, r)\subset X$ with $0<r<{\rm diam}(X)$ and let $y\in B(x, r)\setminus B(x, cr)$. Then
\begin{equation*}\label{bs}
B(y, C_1r)\subset B(x, C_1^{-1}r)\setminus B(x, C_1r).
\end{equation*}

We can construct a sequence of balls $B_1, B_2\ldots \subset B(x,r)$ in the following way. Let 
 $\hat{r}_i=C_1^{2i-1}r$ and $\hat{B}_{i}=B(x, \hat{r}_i)$. By the previous argument, for each $\hat{B}_{i}$, there exists $x_i$ such that 
 \[
 B(x_i, C_1\hat{r}_i)\subset B(x, C_1^{-1}\hat{r}_i)\setminus B(x, C_1\hat{r}_i).
 \]
 Let $r_i=C_1\hat{r}_i=C_1^{2i}r$, so 
\[
 B(x_i, r_i)\subset B(x, r_{i-1})\setminus B(x, r_i).
\]

Let  $C_0=C_1^2$. Clearly, $0<C_0<1$. Let $B_i=B(x_i, r_i)=B(x_i, C_0^{i}r)$. Since these balls are contained in disjoint annuli, $B_i\cap B_j=\emptyset$ for $i\neq j$.
 
From the construction, the balls in this sequence satisfy 
$$B_i\subset B(x, C_0^{i-1}r)\setminus B(x, C_0^{i}r).$$ 
This completes the proof.
\end{proof}

Now, we can complete the proof of the main result follwoing some ideas from \cite{HajD}.
\begin{proof}
By definition, for $u\in M^{1,s}(X, d, \mathcal{H}^{s})$, there exists $g\in L^s(X)$ such that 
\[
|u(x)-u(y)|\le d(x,y)\big(g(x)+g(y)\big),
\]
when $x,y\in X\setminus E$ and $\mathcal{H}^s(E)=0$.

Let $E_0=\{x\in X\setminus E : g(x)<\infty\}$.  Clearly, $\mathcal{H}^s(X\setminus E_0)=0$. Fix a ball $B(z, r)\subset X$. We will prove that
\[
\sup_{x,y\in B(z,r)\cap E_0}|u(x)-u(y)|\le C\Big(\int_{B(z, 2r)}g^s d\Ha\Big)^{\frac{1}{s}}.
\]

Let $x, y\in B(z, r)\cap E_0$. According to Lemma \ref{BS}, there exists a sequence of disjoint balls $\{B_i\}_{i=1}^{\infty}=\{B(x_i, C_0^{i}r)\}_{i=1}^{\infty}\subset B(x, r)$ such that
\[
B_i=B(x_i, C_0^{i}r)\subset B(x, C_0^{i-1}r)\setminus B(x, C_0^{i}r).
\]
Let $r_i=C_0^{i}r$, then
\[
B_i=B(x_i, r_i)\subset B(x, C_0^{-1}r_i)\setminus B(x, r_i).
\]

Clearly, we can find $z_i\in B_i\cap E_0$ such that 
\begin{equation*}
 \begin{aligned}
 g(z_i)\le \Big(\fint_{B_i} g^sd\Ha\Big)^{1/s}&=\Big(\frac{1}{\mathcal{H}^s(B_i)}\Big)^{1/s}\Big(\int_{B_i} g^sd\Ha\Big)^{1/s}\\
 &\le \frac{C}{r_i}\Big(\int_{B_i} g^sd\Ha\Big)^{1/s}.\\
\end{aligned}
 \end{equation*}
Here the barred integral denotes the integral average. 

Notice that $z_i\in B_i\subset B(x, C_0^{-1}r_i)$ implies that $d(x, z_i)\le C_0^{-1}r_i$. It follows that
\begin{eqnarray*}
|u(x)-u(z_i)|&\le& d(x, z_i)\Big(g(x)+g(z_i))\Big) \nonumber \\
&\le& Cr_ig(x)+C\Big(\int_{B_i} g^sd\Ha\Big)^{1/s}.
\end{eqnarray*}
Hence, $\lim_{i\to \infty} u(z_i)=u(x)$. Thus,

\begin{eqnarray*}
|u(z_1)-u(x)|&\le& \sum_{i=1}^{\infty}|u(z_i)-u(z_{i+1})| \nonumber \\
&=& \sum_{i=0}^{\infty}d(z_i, z_{i+1})\Big(g(z_i)+g(z_{i+1})\Big).
\end{eqnarray*}

Since $z_i\in B_i$, $z_{i+1}\in B_{i+1}$ and $B_i, B_{i+1}\subset B(x, C_0^{-1}r_i)$, we have that $d(z_i, z_{i+1})\le Cr_i$. Since also $r_{i+1}=C_0r_i$, the above inequality implies that
\begin{eqnarray}\label{1}
|u(z_1)-u(x)|&\le& \sum_{i=0}^{\infty} Cr_i\Big(\frac{1}{r_i}\big(\int_{B_i} g^sd\Ha\big)^{1/s}+\frac{1}{r_{i+1}}\big(\int_{B_{i+1}} g^sd\Ha\big)^{1/s}\Big) \nonumber \\
&\le & C\sum_{i=0}^{\infty} \Big(\int_{B_i} g^sd\Ha\Big)^{1/s} \nonumber\\
&\le & C\Big(\int_{B(x,r)}g^sd\Ha\Big)^{1/s}.
\end{eqnarray}
The last inequality follows from the reverse Minkowski inequality and the fact that the balls in the sequence $\{B_i\}_{i=1}^{\infty}\subset B(x, r)$ are pairwise disjoint.

Notice that $z_1\in B_1\subset B(x,r)\subset B(z, 2r)$ and 
$$g(z_1)\le \frac{C}{r_1}\Big(\int_{B_1} g^sd\Ha\Big)^{1/s}\le \frac{C}{r}\Big(\int_{B(z, 2r)} g^sd\Ha\Big)^{1/s}.$$
Similarly, we can find $w_1\in B(y, r)\subset B(z, 2r)$ such that 
\begin{equation}\label{2}
|u(w_1)-u(y)|\le C\Big(\int_{B(y,r)}g^sd\Ha\Big)^{1/s},
\end{equation}
and
\[
g(w_1)\le \frac{C}{r}\Big(\int_{B(z,2r)} g^sd\Ha\Big)^{1/s}.
\]

Since $z_1, w_1\in B(z,2r)$, we also have

\begin{eqnarray}\label{3}
|u(z_1)-u(w_1))|&\le & d(z_1, w_1)(g(z_1)+g(w_1)) \nonumber\\
&\le & C\big(\int_{B(z,2r)}g^sd\Ha\big)^{1/s}.
\end{eqnarray}

Now the inequalities \eqref{1}, \eqref{2}, \eqref{3} readily give

\[
\sup_{x,y\in B(z,r)\cap E_0}|u(x)-u(y)|\le C\Big(\int_{B(z, 2r)}g^s d\Ha\Big)^{\frac{1}{s}}.
\]
This and the absolute continuity of the integral imply uniform continuity of $u|_{X\setminus E_0}$. Hence $u$ uniquely extends to a uniformly continuous function on $X$ and the inequality \eqref{0} follows. The proof is complete.

\end{proof}
\section*{Acknowledgments}
The author thanks Professor P. Haj\l{}asz for bringing this topic to her attention and valuable suggestions to improve the paper. She also thanks Professor N. Shanmugalingam for her interests and helpful comments. 
\setcitestyle{square}

\end{document}